\documentclass[11pt]{amsart}

\usepackage{amssymb,amsmath,amscd,graphicx,
latexsym,amsthm}
\usepackage{amssymb,latexsym,amsmath,amscd,graphicx}
\usepackage[mathscr]{eucal}
\usepackage[all]{xy}
\usepackage[a4paper,top=3.5cm,bottom=3.5cm,left=3cm,right=3cm]{geometry}

\usepackage{enumerate}

\usepackage[colorlinks=true, citecolor=blue, urlcolor=blue, linkcolor=blue, pagebackref]{hyperref}

\def\blank{\underline{\hphantom{A}}}

\def\P{\ensuremath{\mathbb{P}}}

\def\cF{\ensuremath{\mathcal F}}

\def\cI{\ensuremath{\mathcal I}}

\def\cL{\ensuremath{\mathcal L}}
\def\cM{\ensuremath{\mathcal M}}

\def\cO{\ensuremath{\mathcal O}}

\def\cX{\ensuremath{\mathcal X}}

\def\oG{\overline G}

\def\tensor{\otimes}

\def\deg{\mathop{\mathrm{deg}}\nolimits}
\def\dim{\mathop{\mathrm{dim}}\nolimits}

\def\Pic{\mathop{\mathrm{Pic}}\nolimits}
\def\rk{\mathop{\mathrm{rk}}}

\def\Hom{\mathop{\mathrm{Hom}}}
\def\det{\mathop{\mathrm{det}}}

\theoremstyle{plain}

\newtheorem{theorem}{Theorem}[section]

\newtheorem{proposition/example}[theorem]{Proposition/Example}
\newtheorem{proposition}[theorem]{Proposition}

\newtheorem{corollary}[theorem]{Corollary}
\newtheorem{lemma}[theorem]{Lemma}

\theoremstyle{definition}

\newtheorem{remark}[theorem]{Remark}

\newtheorem{conjecture/question}[theorem]{Conjecture/Question}

\newtheorem{remark/definition}[theorem]{Remark/Definition}
\newtheorem{notation/assumptions}[theorem]{Assumptions/Notation}

\numberwithin{equation}{section}

 \theoremstyle{remark}



\begin{document}

\title{Stability of syzygy bundles on abelian varieties}

\author{Federico Caucci, Mart\'{i} Lahoz}
 
	\address{F.\ Caucci \\
	Dipartimento di Matematica ``Guido Castelnuovo'', Sapienza Universit\`{a} di Roma, P.le Aldo Moro 5, 00185 Roma, Italy}
\curraddr{Dipartimento di Matematica e Informatica ``Ulisse Dini'', Universit\`{a} di Firenze,
Viale Morgagni 67/a, 50134 Firenze, Italy} 
\email{federico.caucci@unifi.it}

\address{M.\ Lahoz \\
Departament de Matem\`{a}tiques i Inform\`{a}tica, Universitat de Barcelona, Gran Via de les Corts Catalanes, 585, 08007 Barcelona, Spain}
\email{marti.lahoz@ub.edu}
\urladdr{\url{http://www.ub.edu/geomap/lahoz/}}

\thanks{F.~C. was partially supported by the National Group for Algebraic and Geometric Structures, and their Applications (GNSAGA-INdAM).
M.~L. was partially supported by the Spanish MICINN project PID2019-104047GB-I00.}

\maketitle

\setlength{\parskip}{.1 in}

\begin{abstract} 
We prove that the kernel of the evaluation morphism of global sections -- namely the syzygy bundle -- of a sufficiently ample line bundle on an abelian variety is stable. This settles a conjecture of Ein--Lazarsfeld--Mustopa, in the case of abelian varieties.
\end{abstract}

\section{Introduction}
Let $(X, L)$ be a polarized smooth variety, defined over an algebraically closed field $k$.
Suppose that $L$ is globally generated.
The \emph{syzygy} (or \emph{kernel}) \emph{bundle} $M_L$ associated to $L$ is the kernel of the evaluation morphism of global sections of $L$. Hence, by definition, it sits in the short exact sequence 
\begin{equation}\label{eqn:ML}
 0\to M_L\to H^0(X,L)\otimes \cO_X \xrightarrow{ev} L\to 0.
\end{equation}

In recent years
 stability of kernel bundles has been investigated by several authors and the picture is well understood in some cases:

\begin{enumerate}[\rm (1)]
 \item for the projective space $\P^n_k$, it is known that $M_{\cO_{\P^n_k}(d)}$ is slope semistable, if $d > 0$, 
		 by \cite{fl} in $\operatorname{char} (k) = 0$, and \cite{br} in arbitrary characteristic. (See also \cite[Theorem 2.4]{langer2} and the comment below it);
 \item\label{item:curves} for a smooth projective curve of genus $g \geq 1$, $M_L$ is semistable if $\deg L \geq 2g$ (see \cite{einla}).
\end{enumerate}

 More recently, Ein--Lazarsfeld--Mustopa \cite{einlamu} -- based on previous results of Camere \cite{ca} -- proved, for smooth projective surfaces, the slope stability of $M_{L^d}$ with respect to $L$ for $d$ sufficiently large and, in arbitrary dimension, they obtained the same for varieties with Picard group generated by $L$, strengthening an argument of Coand\v{a} for the projective space (see \cite{co} and \cite[Proposition~C]{einlamu}). In \cite{einlamu}, the authors also conjectured that the same result should hold for \emph{any} smooth projective variety.

 Note that if 
\[
\varphi_{|L|} : X \rightarrow \P := \P(H^0(X, L)^{\vee}) 
\]
is the morphism associated to $L$, then the slope (semi)stability of $M_L$ with respect to $L$ is equivalent to the slope (semi)stability of the pull back $\varphi_L^* T_{\P}$ of the tangent bundle of $\P$.
Moreover, note that slope (semi)stability with respect to $L$ only depends on the numerical class of $L$ up to a positive real or rational multiple.

The first result of this short note is the following

\begin{theorem}\label{thm:stability}
Let $(X,L)$ be a polarized abelian variety defined over an algebraically closed field $k$ and let $d\geq 2$.
Then the syzygy bundle $M_{L^d}$ is semistable with respect to $L$.
\end{theorem}

 Recall that, if $(X,L)$ is a polarized abelian variety, then $L^d$ is globally generated for any $d\geq 2$.

Theorem~\ref{thm:stability} recovers the classical case of elliptic curves (see \eqref{item:curves} above) and it solves in the affirmative the aforementioned \cite[Conjecture~2.6]{einlamu} in the case of abelian varieties (see Remark~\ref{rem:slopestab_simple} below).
For complex abelian surfaces, Camere (\cite{ca}) proved that $M_L$ is slope stable, if $L$ is base point free and $h^0(X, L) \geq 7$.

\begin{remark}[Slope stability]\label{rem:slopestab_simple}
If $g!$ divides $d^{g-1}$, where $g=\dim X$, then the rank and the degree of $M_{L^d}$ are coprime, so semistability coincides with slope stability.
\end{remark}

\begin{remark}[Positive characteristic]
If $\operatorname{char} (k) =p>0$ and $d\geq 2$, then we can conclude that $M_{L^d}$ is \emph{strongly} slope semistable by \cite[Theorem~2.1]{mera} (see also \cite[Section 6]{langer}).
\end{remark}

\begin{remark}[Higher-order kernel bundles]
Let $C$ be a smooth projective curve of genus $g$, defined over an algebraically closed field of characteristic $0$. Let $L$ be an ample line bundle on $C$.
If $i$ is a non-negative integer,
 $P^i(L)$ denotes the bundle of $i$-th order principal parts of $L$, so that $P^i(L)$ is a vector bundle of rank $i +1$ on $C$. If $\deg L \geq 2g + i$, then the natural morphism
\[
H^0(C, L) \otimes \cO_C \rightarrow P^i(L) 
\]
is surjective, hence its kernel is a vector bundle denoted by $R^i(L)$. From a geometric perspective, if $L$ embeds $C$ into the projective space $\P = \P(H^0(C, L)^{\vee})$, then $R^0(L)^{\vee} \otimes L = T_{\P}|_{C}$, and $R^1(L)^{\vee} \otimes L = N_{C/\P}$, the normal bundle of $C$ in $\P$. In \cite{einla}, Ein and Lazarsfeld proved that, for an \emph{elliptic} curve, $R^i(L)$ is semistable if $\deg L \geq 2 + i$, and they conjectured that a similar result should hold for higher genus curves. Recently, in \cite{einniu}, Ein and Niu give a strong evidence for this conjecture, by using the notion of interpolation for vector bundles on curves. 
All this suggests that it could be possible to extend Ein--Lazarsfeld result to abelian varieties of arbitrary dimension. If $(X, L)$ is a polarized abelian variety of dimension $g$, then, for $d \geq 2 + i$, one has the short exact sequence
\[
0 \rightarrow R^i(L^d) \rightarrow H^0(X, L^d) \otimes \cO_X \rightarrow P^i(L^d) \rightarrow 0,
\]
where now $P^i(L^d)$ is a vector bundle of rank ${i+g}\choose{i}$. Notice that $R^0(L^d) = M_{L^d}$ and,
 in light of our result, it is natural to wonder if $R^i(L^d)$ is semistable with respect to $L$, when $d \geq 2 + i$. 
\end{remark}

Before proving Theorem~\ref{thm:stability}, we first establish the following result, of independent interest since it considers the case of primitive line bundles.
\begin{theorem}\label{newthm}
Let $(X,L)$ be a polarized abelian variety of dimension $g$, defined over an algebraically closed field $k$. If $L$ is globally generated and for any non-zero abelian subvariety $Z \subseteq X$ we have
\[
(L^{\dim Z} \cdot Z) \geq \frac{g!}{(g- \dim Z)!},
\]
then $M_L$ is slope stable with respect to $L$.
\end{theorem} 
In particular, on a \emph{simple} abelian variety, $M_L$ is slope stable as soon as $L$ is globally generated (see Corollary~\ref{prop:stability_simple}). 

The proof of Theorem~\ref{thm:stability} goes as follows: we consider the particular case of simple abelian varieties of Theorem~\ref{newthm} (see Corollary~\ref{prop:stability_simple}); then, since polarized simple abelian varieties are dense in their moduli space (Proposition~\ref{prop:density}), we use the properness of the relative moduli space of semistable sheaves, in order to get a semistable sheaf on $X$, that turns out to be isomorphic to the original kernel bundle.

\subsection*{Notation}
We will freely use the notation and terminology of \cite{hule}.
In particular, a (semi)stable sheaf is a Gieseker (semi)stable sheaf. 
When the line bundle $L$ is clear from the context, for any $d>0$ we may use the notation $M_d:=M_{L^d}$.

\subsection*{Acknowledgements} 
We thank Gerald Welters for pointing us the argument for the density of simple abelian varieties in positive characteristic.
We would also like to thank the anonymous referee for suggesting improvements to the exposition.
This work was started during a visit of the first named author at Universitat de Barcelona in 2019, and it is part of his PhD thesis: he thanks the Institute of Mathematics of the University of Barcelona (IMUB) for its hospitality, and his advisor, Giuseppe Pareschi, for his teachings.

\section{Stability of syzygy bundles}
Let $(X,L)$ be a polarized abelian variety of dimension $g$.
Suppose that $L$ is globally generated. From the exact sequence \eqref{eqn:ML}, we get 
\begin{align*}
c_1 (M_L) &= - c_1 (L), \\ 
 \rk(M_L) &= h^0(X, L) - 1.
\end{align*}
Therefore, using Riemann-Roch and the fact that $L$ has no higher cohomology (see \cite[\S 16]{mu}), the slope of $M_L$ with respect to $L$ is 
\[
\mu_L(M_L) = -\frac{g! \chi(X,L)}{\chi(X,L)-1}.
\]

\subsection{Stability for primitive line bundles}

In this section we prove Theorem~\ref{newthm}. Actually,
in order to prove Theorem~\ref{thm:stability}, we only need the particular case of simple abelian varieties in Theorem~\ref{newthm}. Namely,
\begin{corollary}\label{prop:stability_simple}
Assume $X$ simple and $L$ globally generated.
Then the syzygy bundle $M_L$ is slope stable with respect to $L$.
\end{corollary}

Recall the following well-known lemma:
\begin{lemma}\label{lem:cohstab}
Let $(X,L)$ be a polarized $n$-dimensional smooth variety over $k$.
Let $E$ be a vector bundle on $X$.
Suppose that for any integer $r$ and any line bundle $G$ on $X$ such that
\begin{equation}\label{eqn:cohstab}
 0 < r < \rk(E)\quad \text{ and }\quad (G \cdot L^{n-1}) \geq r\, \mu_L(E)
\end{equation}
one has 
\[
H^0(X, \Lambda^r E\, \tensor \, G^{\vee}) = 0.
\]
Then $E$ is slope stable with respect to $L$.
\end{lemma}

 If $E$ satisfies the condition of Lemma~\ref{lem:cohstab}, it is said to be \emph{cohomologically stable} with respect to $L$. 
In order to prove that a kernel bundle is cohomologically stable, following Coand\v{a} \cite{co}, we make use of a vanishing result of M.~Green:
\begin{lemma}[{\cite[Theorem~3.a.1]{green}, see also \cite[Lemma~2.2]{co}}]\label{lem:green}
Let $L$ and $Q$ be line bundles on a smooth projective variety $X$, with $L$ globally generated. If $r \geq h^0(X, Q)$, then 
\[
H^0(X, \Lambda^r M_{L} \tensor Q) = 0.
\] 

\end{lemma}

\begin{proof}[Proof of Theorem~\ref{newthm}]
We want to prove that $M_L$ satisfies the hypothesis of Lemma~\ref{lem:cohstab}, that is, $M_L$ is cohomologically stable.
Let $r$ and $G$ satisfying Condition \eqref{eqn:cohstab}.
If $h^0(X, G^{\vee}) \leq 1$, then, by Green's vanishing Lemma~\ref{lem:green}, we are done.
Hence, we assume 
\[
h^0(X, G^{\vee}) > 1.
\] 
Consider the usual morphism of abelian varieties
\[
\phi_{G^\vee} \colon X \rightarrow \widehat{X}, \quad x \mapsto t_x^*(G^\vee) \tensor G
\] 
induced by $G^\vee$, where $t_x \colon X \rightarrow X$ is the translation morphism by the element $x \in X$. The reduced connected component of the kernel $K(G^\vee)$ of 
$\phi_{G^\vee}$ containing $0$ is an abelian subvariety of $X$, denoted by $K(G^\vee)^0_{\mathrm{red}}$. 
Note that $K(G^\vee)^0_{\mathrm{red}} \neq X$, because otherwise $G^\vee \in \Pic^0 X$, hence $h^0(X, G^\vee) \leq 1$.
Let 
\[
p \colon X \rightarrow Y := X/K(G^\vee)^0_{\mathrm{red}} 
\]
be the non-trivial quotient to the abelian variety $Y$
of dimension $k = g - \dim K(G^\vee)^0_{\mathrm{red}} > 0$.
Being $G^\vee$ an effective line bundle, we have that the restriction $G^\vee|_{K(G^\vee)^0_{\mathrm{red}}}$ is trivial. Hence, $h^0(K(G^\vee)^0_{\mathrm{red}}, G^\vee|_{K(G^\vee)^0_{\mathrm{red}}}) = 1$ and, 
by cohomology and base change, $\oG^\vee := p_*G^\vee$ is a line bundle on $Y$ such that 
\[
p^* \oG^\vee = G^\vee.
\] 
Moreover, by the projection formula
\[
h^0(X, \oG^\vee) = h^0(X, G^\vee),
\]
and, since $K(\oG^\vee)$ is finite, $\oG^\vee$ is ample (see \cite[p.\ 60]{mu}). In particular, $\chi(X, \oG^\vee) = h^0(X, \oG^\vee)$.
Note that if $G^\vee$ was already ample, we are just considering $Y=X$ and $k=g$.

Now, we can rewrite the inequality $(G\cdot L^{g-1})\geq r\, \mu_L(M_L)$ as 
\[
 r\geq \frac{ \chi(X,L)-1}{g! \chi(X,L)}(G^{\vee}\cdot L^{g-1}).
\]
By the generalized Hodge-type inequality
\cite[Theorem~1.6.3(i) and Remark~1.6.5]{laI}
\begin{align*}
(G^{\vee}\cdot L^{g-1})
&\geq \sqrt[k]{((G^{\vee})^k\cdot L^{g-k}) (L^g)^{k-1}}\\
&=\sqrt[k]{((G^{\vee})^k\cdot L^{g-k}) (g!\chi(X,L))^{k-1}}\\
&=\sqrt[k]{((\oG^{\vee})^k\cdot p_*L^{g-k})\cdot(g!\chi(X,L))^{k-1}}\\
&=\sqrt[k]{(L^{g-k}\cdot K(G^\vee)^0_{\mathrm{red}})\cdot (\oG^{\vee})^k\cdot (g!\chi(X,L))^{k-1}}\\
&=\sqrt[k]{(L^{g-k}\cdot K(G^\vee)^0_{\mathrm{red}})\cdot k!h^0(X,G^\vee) \cdot (g!\chi(X,L))^{k-1}},
\end{align*}
where in the second to last equality we have used that $L$ is ample, so when restricted to $K(G^\vee)^0_{\mathrm{red}}$ is not trivial. 
Moreover, $[L]^{g-k}$ is an effective $k$-cycle on $X$ and $p$ restricted to it is finite and has degree $(L^{g-k}\cdot K(G^\vee)^0_{\mathrm{red}})$.
Now, by hypothesis,
 \[
 (L^{g-k}\cdot K(G^\vee)^0_{\mathrm{red}})\geq \tfrac{g!}{k!}.
 \]
Then
\begin{align*}
 (G^{\vee}\cdot L^{g-1})\geq g! \sqrt[k]{ h^0(X,G^\vee)\cdot \chi(X,L)^{k-1}}. 
\end{align*}
So we obtain
\begin{equation}\label{eqn:risc}
r\geq\frac{\chi(X,L)-1}{g! \chi(X,L)}(G^{\vee}\cdot L^{g-1})\geq \frac{\chi(X,L)-1}{\chi(X,L)}\sqrt[k]{h^0(X, G^{\vee}) \chi(X, L)^{k-1} }.
\end{equation}
If $h^0(X,G^{\vee})\geq \chi(X,L)$, then \eqref{eqn:risc} becomes 
 $r\geq \chi(X,L)-1$ contradicting $r<\rk(M_L)$ in Condition \eqref{eqn:cohstab}. Thus $\chi(X,L)>h^0(X,G^{\vee})$ and \eqref{eqn:risc} becomes
\[
 r> \frac{\chi(X, L) - 1}{\chi(X, L)} h^0(X, G^{\vee}) = h^0 (X,G^{\vee})-\frac{h^0(X,G^{\vee})}{\chi(X,L)}.
\]
Since $\chi(X,L)>h^0(X,G^{\vee})$ and $r$ is an integer, this is equivalent to $r\geq h^0(X,G^{\vee})$ and, by Green's vanishing Lemma~\ref{lem:green}, we are done.
\end{proof}

Note that the proof is valid in any characteristic.

\subsection{Proof of Theorem~\ref{thm:stability}}

We start with the following 
\begin{proposition}\label{prop:density}
Given $g>0$ and $k$ an algebraically closed field, the set of $(d_1,\ldots,d_g)$-polarized simple abelian varieties of dimension $g$ over $k$ is dense in its moduli space $\mathcal{A}_ g(k)$.
\end{proposition}

\begin{proof}
For fields $k$ of characteristic $0$, this is a classical result (see, e.g., \cite[Exercise~8.11.(1)]{bila}).
In positive characteristic, we can argue as follows.
Given $(X, L)$ an \emph{ordinary} $(d_1,\ldots,d_g)$-polarized abelian variety, the set of $(X', L')$ isogenous to $(X,L)$ is dense in the moduli space $\mathcal A _ g(k)$ (see \cite[Theorem~2, p.\ 477]{ch}).
Hence, if we can produce an ordinary and simple abelian variety $(X, L)$ over $k$, we are done.
In order to achieve that, we can reduce ourselves to the case of $k=\overline{\mathbb{F}}_p$, since both properties are preserved by base change from $\overline{\mathbb{F}}_p$ to our starting $k$.
Finally, from \cite{hozh} we get the existence of ordinary and geometrically simple abelian varieties of any dimension over $\mathbb{F}_p$.
\end{proof}

Let $(X,L)$ be a polarized non-simple abelian variety of dimension $g$.
Consider $\cX\to T$ a family of abelian varieties polarized by a relatively ample line bundle $\cL$, such that $(\cX,\cL)_{0}\cong (X,L)$, for $0 \in T$.
We will denote by 
\[
 S:=\{ s \in T \ |\ \cX_s \textrm{ is a simple abelian variety} \}
\]
the corresponding dense subset in $T$.

Now we are ready to give the proof of Theorem~\ref{thm:stability}.
By Corollary~\ref{prop:stability_simple}, for any $d\geq 2$ and any polarized simple abelian variety $(\cX_s,\cL_s)$, the corresponding syzygy bundle $M_{d}$, as defined in \eqref{eqn:ML}, is slope stable with respect to $\cL_s$.
Let $P_d(m):=\chi(M_{L^d}\otimes L^m)$ be the Hilbert polynomial of $M_{L^d}$. 
By \cite[Theorem~0.2]{langer}, we have a projective relative moduli space $\cM_{\cX/T}(P_d)\to T$.
Since 
\[
M_d\in \cM_{\cX_s}(P_d)\cong M_{\cX/T}(P_d)_s
\]
for any $s\in S$ and $S$ is dense in $T$, there is a family $\cF\in \cM_{\cX/T}(P_d)$ such that $\cF_s=M_{\cL_s^d}$ for $s\in S$, by the properness of the relative moduli space.
Since $S$ is dense, we have 
\[
\dim \Hom(\cF_{0},\cO_{X})\geq h^0(\cX_s, M_{\cL_s^d}^{\vee}) = h^0(\cX_s, \cL_s^d) = h^0(X,L^d),
\]
by semicontinuity.
Thus, we can consider $\psi\colon \cF_{0} \to H^0(X,L^d) \tensor \cO_X$ generically of maximal rank, which is injective, since $\cF_{0}$ is torsion-free.
Then, we can consider the following commutative diagram with exact rows 
\begin{equation} \label{labdiagram} \vcenter {
\xymatrix{
0 \ar[r] & \cF_{0} \ar[r]^-{\psi} \ar@{^{(}->}[d] & H^0(X,L^d) \tensor \cO_X \ar[r] \ar@{=}[d] & Q \ar[r] \ar[d] & 0 \\ 
 0 \ar[r] & \cF_{0}^{\vee\vee} \ar[r] & H^0(X,L^d) \tensor \cO_X \ar[r]^-{\eta} & Q^{\vee\vee} }
}
\end{equation}
where $Q:=\operatorname{coker} \psi$ has the same numerical class of $L^d$.
Let $T(Q) \subset Q$ be the torsion subsheaf of $Q$, and take the resulting short exact sequence
\begin{equation}\label{torsion1}
0 \rightarrow T(Q) \rightarrow Q \rightarrow Q' := Q/T(Q) \rightarrow 0.
\end{equation}
We have that $T(Q) = \ker[Q \rightarrow Q^{\vee \vee}] \cong \operatorname{coker}[\cF_0 \hookrightarrow \cF_0^{\vee \vee}]$ has codimension greater or equal to $2$, because $\cF_0$ is torsion-free. 
Therefore, $\det(T(Q)) = \cO_X$ and $\det(Q) = \det(Q')$ (see, e.g., \cite[Chap.\ V, (6.14) and (6.9)]{ko}). Moreover, by applying $\mathcal{H}om(\blank, \cO_X)$ to (\ref{torsion1}), we get 
\[
0 \rightarrow (Q')^{\vee} \rightarrow Q^{\vee} \rightarrow T(Q)^{\vee}, 
\] 
and $T(Q)^{\vee} = 0$ (see, e.g., \cite[Proposition~1.1.6(i)]{hule}).
Hence, $(Q')^{\vee \vee} \cong Q^{\vee \vee}$ and $\det(Q') = (\det(Q'))^{\vee \vee} = \det((Q')^{\vee \vee}) = \det(Q^{\vee \vee})$, where in the second equality we used that $Q'$ is torsion-free (see, e.g., \cite[Chap.\ V, (6.12)]{ko}).
Being a reflexive sheaf of rank $1$, $Q^{\vee \vee}$ is a line bundle, and, by the above discussion, it is algebraically equivalent to $L^d$, say $Q^{\vee\vee} = L^d \otimes \alpha = t_x^* L^d$, for some $\alpha \in \Pic^0{X}$ and $x \in X$, where $t_x : X \rightarrow X$ is the translation by $x$.
Thus the map $\eta$ in (\ref{labdiagram}) is surjective. If not, indeed, we would have 
\[
0\to \cF_0^{\vee\vee}\to H^0(X,L^d)\otimes \cO_X\xrightarrow{\eta} \cI_Z\otimes t_x^*L^d\to 0,
\]
where $Z \subseteq X$ is a closed subscheme.
From our hypothesis on $d$, the (pullback along $t_x$ of the) evaluation map for $L^d$ is surjective, so the map $\eta$ is forced to factor via a non-trivial linear quotient $V$ of $H^0(X,L^d)$:
\[
\xymatrix{
0\ar[r] & \cF_{0}^{\vee\vee} \ar[r] & H^0(X,L^d) \tensor \cO_X \ar[r]^-{\eta} \ar[d] & \cI_Z\otimes t_x^*L^d \ar[r] \ar@{=}[d] & 0 \\ 
 & & V \tensor \cO_X \ar[r] & \cI_Z\otimes t_x^*L^d\ar[r]& 0. }
\]
Hence, if we denote $W:= \ker [H^0(X,L^d) \rightarrow V]$, then $W\otimes \cO_X\hookrightarrow \cF_0^{\vee\vee}$. This contradicts the fact that $\cF_0^{\vee\vee}$ is slope semistable with respect to $L$, because it has negative slope.
From the commutativity of the diagram (\ref{labdiagram}), we get that the map $Q \rightarrow Q^{\vee \vee}$ is surjective too, and, since $Q$ and $Q^{\vee \vee}$ have the same numerical class, it is also injective. Therefore, $Q\cong Q^{\vee\vee}$ and $\cF_0 \cong \cF_0^{\vee \vee} \cong t_x^* M_{L^d}$.
In particular, $M_{L^d}$ is semistable.
\qed

\end{document}